\RequirePackage{xr-hyper}
\documentclass[letterpaper, 12pt]{article}
\pdfoutput = 1

\usepackage[english]{babel}
\usepackage[utf8]{inputenc}

\usepackage{microtype}
\usepackage{libertineRoman}

\usepackage[margin = 1.2 in, top = 1 in, bottom = 1.2 in]{geometry}
\makeatletter
\g@addto@macro\normalsize{%
  \setlength\abovedisplayskip{7pt}
  \setlength\belowdisplayskip{7pt}
  \setlength\abovedisplayshortskip{7pt}
  \setlength\belowdisplayshortskip{7pt}
}
\makeatother

\usepackage{amsfonts}
\usepackage{mathrsfs}
\usepackage{bbm}
\usepackage{latexsym}
\usepackage{math dots}
\usepackage{amssymb}
\usepackage{mathtools}

\usepackage{enumitem}
\setlist{nolistsep}
\usepackage{amsthm}

\usepackage{xcolor}
\definecolor{Color1}{rgb}{0.0, 0.42, 0.47}
\definecolor{Color2}{rgb}{0.78, 0.11, 0.0}

\usepackage{titlesec}
\titleformat{\section}
  {\large\center\bfseries}
  {\thesection.}{.7em}{}
\titlespacing*{\section}{0pt}{3.5ex plus 0ex minus 0ex}{1.5ex plus 0ex}
\titleformat{\subsection}
  {\center\bfseries}
  {\thesubsection.}{.7em}{}
\titlespacing*{\subsection}{0pt}{3.5ex plus 0ex minus 0ex}{1.5ex plus 0ex}
\titleformat{\subsubsection}
  {\center\bfseries}
  {\thesubsubsection.}{.7em}{}
\titlespacing*{\subsubsection}{0pt}{3.5ex plus 0ex minus 0ex}{1.5ex plus 0ex}

\usepackage{titling}
\makeatletter
\renewenvironment{abstract}{
    \begin{center}
        {\bfseries \large\abstractname\vspace{\z@}}
    \end{center}
    \quotation
}
{\endquotation}
\makeatother

\usepackage{hyperref}
\usepackage[capitalize]{cleveref}
\hypersetup{citecolor = Color2,colorlinks,
			linkcolor = black,
			urlcolor = Color2,
            pdfauthor = {Bora \c{C}al{\i}m, Ioannis Iakovakis, Sophie Long, Jack Moffatt, Deborah Wooton},
            pdftitle = {Popular differences in primes along fractional powers}}

\newtheoremstyle{plain}{3mm}{3mm}{\slshape}{}{\bfseries}{.}{.5em}{}
\newtheoremstyle{definition}{2mm}{2mm}{}{}{\bfseries}{.}{.5em}{}
\theoremstyle{plain}
	
\newtheorem{Theorem}{Theorem}

\newtheorem{Lemma}[Theorem]{Lemma}
\newtheorem{Proposition}[Theorem]{Proposition}
\newtheorem{Corollary}[Theorem]{Corollary}
\newtheorem{Conjecture}[Theorem]{Conjecture}
\theoremstyle{definition}
\newtheorem{Definition}[Theorem]{Definition}

\theoremstyle{plain} 
\newcounter{MainTheoremCounter}

\theoremstyle{plain}
\newtheorem*{namedthm}{\namedthmname}
\newcounter{namedthm}
	

\counterwithin{Theorem}{section}
\numberwithin{equation}{section}

\allowdisplaybreaks

\usepackage{doi}
\usepackage{csquotes}
\usepackage[backend = biber, style = alphabetic, maxbibnames = 99, maxcitenames = 99]{biblatex}
\newcommand{\Cesaro}{Ces\`{a}ro}
\newcommand{\Erdos}{Erd\H{o}s}

\newcommand{\Szemeredi}{Szemer\'{e}di}
\newcommand{\Sarkozy}{S\'{a}rk\"{o}zy}
\newcommand{\Turan}{Tur{\'a}n}

\newcommand{\Oh}{{\rm O}}
\newcommand{\oh}{{\rm o}}
\newcommand{\N}{\mathbb{N}}
\newcommand{\Z}{\mathbb{Z}}
\newcommand{\R}{\mathbb{R}}
\newcommand{\C}{\mathbb{C}}

\newcommand{\define}[1]{{\itshape #1}}

\renewcommand{\epsilon}{\varepsilon}
\renewcommand{\leq}{\leqslant}
\renewcommand{\geq}{\geqslant}

\renewcommand{\P}{\mathbb{P}}

\newcommand{\E}{\mathbb{E}}
\newcommand{\1}{\textbf{1}}

\DeclarePairedDelimiter\ceil{\lceil}{\rceil}
\DeclarePairedDelimiter\floor{\lfloor}{\rfloor} 
\DeclareMathOperator{\Leb}{Leb}
\DeclareMathOperator{\e}{e}

\usepackage{tikz}
\usetikzlibrary{graphs, quotes}
\usepackage{float}

\title{\bfseries Popular differences in primes along fractional powers}
\date{\small{November 26, 2024}}
\author{By {\scshape Bora \c{C}al{\i}m}, {\scshape Ioannis Iakovakis}, {\scshape Sophie Long,}\\
{\scshape Jack Moffatt}, and {\scshape Deborah Wooton}}

\begin{document}

\maketitle

\begin{abstract}
    We prove that $\mathop{\E}_{m \leq M} \mathop{\E}_{n \leq N} \Lambda(n) \Lambda\bigl(n + \floor{m^c}\bigr) = 1 + \Oh(\log^{2 - Bc} N)$, where $c > 2$ is a non-integer, $B \geq 3/c$, and $M$ is of order $N^{1/c} \log^{-B} N$. As a combinatorial consequence, we obtain that the primes contain infinitely many pairs whose difference belongs to the Piatetski-Shapiro sequence $\bigl\{\floor{m^c} \colon m \in \N \bigr\}$ for any non-integer $c > 2$.
\end{abstract}

\section{Introduction}

In the '70s, \Sarkozy{} \cite{sarkozy1978} proved that any subset of integers with positive upper density contains two elements differing by a perfect square. In fact, he proved that if a set does not contain two elements differing by a perfect square, then its density up to $N$ is $\Oh\bigl((\log\log N)^{2/3} (\log N)^{-1/3}\bigr)$. Although \Sarkozy{}'s result does not cover the case of the primes, a better quantitative bound was later given by Pintz, Steiger, and \Szemeredi{} \cite{pintzSteigSzemeredi1988}, which allows one to conclude that the primes contain infinitely many square differences due to density reasons alone. This bound has since been refined; see \cite{bloomMaynard2022}.

As it turns out, one can replace $P(x) = x^2$ in \Sarkozy{}'s theorem with any intersective polynomial, i.e., any polynomial that has a root modulo every prime. The best quantitative bound in this case is due to Rice \cite{rice2019}, which implies that the primes (and any set with density comparable to the primes) also contain infinitely many pairs of elements differing by a number in the image of any intersective polynomial.

Once one knows that the primes exhibit polynomial differences, it becomes natural to ask about the frequency with which these patterns occur. This question was first addressed by Tao and Ziegler (\cite{taoZiegler2008}, \cite{taoZiegler2018}), who provided asymptotic estimates for a wide range of polynomial patterns in the primes. As a special case\footnote{The constant term of 1 in \cref{eqn:tao-ziegler} can be deduced from \cite[Theorem 4]{taoZiegler2018}.}, they proved that
\begin{equation}\label{eqn:tao-ziegler}
    \mathop{\E}_{m \leq M} \mathop{\E}_{n \leq N} \Lambda(n) \Lambda(n + m^k) 
    = 1 + \oh_{N\to\infty}(1),
\end{equation}
with $M = \oh(N^{1/k})$ and $M \geq \omega(N)N^{1/k}$ for some function $\omega$ going to zero sufficiently slowly depending on $k$. Here, $\E$ stands for the \Cesaro{} average (see \cref{def:cesaro}) and $\Lambda$ denotes the von Mangoldt function (see \cref{def:von-mangoldt}).

In this paper, we investigate
pairs of primes differing by a fractional power instead of an integer power. For this we consider a variant of \cref{eqn:tao-ziegler} of the form
\begin{equation}\label{eqn:avg}
    \mathop{\E}_{m \leq M} \mathop{\E}_{n \leq N} \Lambda(n) \Lambda\bigl(n + \floor{m^c} \bigr),
\end{equation}
where $\floor{m^c}$ is a Piatetski-Shapiro sequence for some non-integer $c>2$.

To our knowledge, averages as in \cref{eqn:avg} have not been considered before, and it is not known whether the primes contain infinitely many pairs whose difference is of the form~$\floor{m^c}$.
However, based on the heuristic that the map $m \mapsto \floor{m^c}$ nullifies multiplicative structure -- a principle supported by numerous examples in the literature, albeit in different contexts (e.g., \cite{delmer2002}) -- it is natural to conjecture that the average in \cref{eqn:avg} behaves similarly to the Tao-Ziegler case.
Our main result confirms this prediction.

\begin{Theorem}\label{thm:main-thm}
    Let $c > 2$ be a non-integer.
    For any $A \geq 1$ , if $B = \frac{A + 2}{c} $ and $M = \frac{N^{1/c}}{\log^B N}$, then
    \[
        \mathop{\E}_{m \leq M} \mathop{\E}_{n \leq N} \Lambda(n) \Lambda\bigl(n + \floor{m^c} \bigr) 
         = 1 + \Oh\biggl(\frac{1}{\log^A N} \biggr).
    \]
\end{Theorem}

This theorem immediately implies the following combinatorial corollary.
\begin{Corollary}\label{cor:infinitely-many-pairs}
    Let $c > 2$ be a non-integer. Then there are infinitely many pairs $(p, m)$ such that $p$ and $p + \floor{m^c}$ are both prime.
\end{Corollary}

In fact the theorem gives $\approx NM \log^{-2} N$ many such pairs with $p \leq N, m \leq M$, which is the expected asymptotic if we assume ``independence" of $\Lambda(n)$ and $\Lambda\bigl(n + \floor{m^c}\bigr)$ (see \cref{notation} for the precise meaning of \(\approx\)). The restriction $c > 2$ is due to technical reasons: a lemma used in \cite{poulias2021} breaks down. We expect the result to hold with $c > 0$, but this is unproven as of yet. Note that when $0 < c < 1$, we trivially get infinitely many pairs of primes differing by some integer of the form $\floor{m^c}$, since the image of $\floor{m^c}$ is all positive integers.

\subsection{Acknowledgments}
We would like to thank the Bernoulli Center for funding the Young Researchers in Mathematics Summer Program. We are also grateful to Florian Richter and Felipe Hern{\'a}ndez for their mentorship and guidance throughout the program. Finally, we would like to thank Nadia Kaiser for her administrative assistance in organizing our travel and stay, as well as those who contributed anonymous feedback.

\section{Preliminaries}

\subsection{Notation}\label{notation}

\hspace{20pt} In this paper, we use the $\Oh(\cdot)$ and $\oh(\cdot)$ notations. For functions $f, g \colon \R \to \C$, $f(x) = \Oh\bigl(g(x)\bigr)$ means $\lvert f \rvert \leq Cg$ for some constant $C$, and $f(x) = \oh\bigl(g(x)\bigr)$ means $\lvert f/g \rvert \to 0$ as $x \to \infty$. We also use the Vinogradov notations $\lesssim$ and $\approx$, where $f \lesssim g$ means $f = \Oh(g)$ and $f \approx g$ means $f \lesssim g \lesssim f$. The notation $\floor{x}$ refers to the floor function, the largest integer $n$ such that $n \leq x$. Similarly, the notation $\{x\}$ refers to the fractional part of $x$, given by $x - \floor{x}$.

We define the \define{\Cesaro{} average} of a function $f$ over positive integers $n \leq x$ as 
\begin{equation}\label{def:cesaro}
    \mathop{\E}_{n \leq x} f(n) = \frac{1}{\floor{x}} \sum_{n = 1}^{\floor{x}} f(n).
\end{equation}
Finally, \(\e(\cdot)\) denotes the function \(\e(x) = e^{2 \pi i x}\) and \(\1_{X}(\cdot)\) denotes the indicator function of the set \(X\), i.e.,
\[
    \1_{X}(x) = \begin{cases}
        1 &\text{if } x \in X\\
        \hfil 0 &\text{else}
    \end{cases}.
\]

\subsection{Basic Notions}

We begin with some definitions and standard results regarding Fourier analysis on $\Z/N\Z$ for some $N \in \mathbb{N}$. 

\begin{Definition}
    The \define{discrete Fourier transform} of a function $f \colon \Z/N\Z \to \C$ is defined for all $\xi \in \Z/N\Z$ as
    \[
        \hat{f}(\xi) = \sum_{n \in \Z/N\Z} f(n) \e\biggl(\! -\frac{\xi n}{N} \biggr).
    \]
\end{Definition}

The following two theorems are standard number theoretic results.

\begin{Theorem}[Fourier inversion]\label{thm:fourier}
    Any function $f \colon \Z/N\Z \to \C$ can be fully recovered from its Fourier coefficients using the formula
    \[
        f(n)=\frac{1}{N}\sum_{\xi \in \mathbb{Z}/N\mathbb{Z}} \hat{f}(\xi) \e\biggl(\frac{\xi n}{N}\biggr).
    \]
\end{Theorem}

\begin{Theorem}[Parseval's identity]\label{thm:parseval}
    For any $f \colon \Z/N\Z \to\C$,
    \[
        \sum_{n \in \Z/N\Z} \lvert f(n) \rvert^2 = \frac{1}{N} \sum_{\xi \in \Z/N\Z} \bigl\lvert \hat{f}(\xi) \bigr\rvert^2.
    \]
\end{Theorem}

\begin{Definition}\label{def:von-mangoldt}
    The \define{von Mangoldt function} \(\Lambda \colon \Z \to \C\) is given by
    \[
        \Lambda(n) =
        \begin{cases}
            \log(p) &\text{if } n = p^k, \text{ where } p \text{ is prime and } k \geq 1 \text{ is an integer}\\
            \hfil 0 &\text{else}
        \end{cases}.
    \]
    We will denote the Fourier transform of the von Mangoldt function restricted to \(\{1, \ldots, N\}\) and viewed as a function on \(\Z/N\Z\) by \(\hat{\Lambda}\).
\end{Definition}

The following estimate, equivalent to the Prime Number Theorem, will prove useful.
\begin{Theorem}[\protect{\cite[Theorem 6.9]{montgomeryVaughan2006}}]\label{thm:prime-number-equiv}
   We have
    \[
        \frac{1}{N} \sum_{n = 1}^{N} \Lambda(n) = 1 + \Oh \biggl(\frac{1}{e^{t \sqrt{\log N}}} \biggr),
    \]
    where $t$ is a positive constant.
\end{Theorem}

In subsequent sections, we will use the following result of Vinogradov about asymptotics for the Fourier transformation of the von Mangoldt function.

\begin{Theorem}[Vinogradov \protect{\cite[p. 143]{davenport1967}}]\label{thm:vinogradov}
    There exists a constant $C$ such that the following holds: If there exist $q \in \{1, \ldots, N\}$ and $a \in \{0, 1, \ldots, q - 1\}$ such that $\gcd(a,q) = 1$ and $\bigl\lvert \frac{\xi}{N} - \frac{a}{q} \bigr\rvert \leq \frac{1}{q^2}$, then
    \[
        \lvert \hat{\Lambda}(\xi) \rvert \leq C \biggl(\frac{N}{\sqrt{q}} + N^{\frac{4}{5}} + \sqrt{qN} \biggr) \log^4 N.
    \]
\end{Theorem}

We will also use the following inequality due to \Erdos{} and \Turan, which gives a bound on how ``far away" a measure is from the Lebesgue measure in terms of its Fourier coefficients.

\begin{Theorem}[\Erdos-\Turan{} inequality \protect{\cite[Theorem 3]{erdosTuran1948-1, erdosTuran1948-2}}]\label{thm:erdos-turan}
    Let $\mu$ be a probability measure on the unit circle \(S^1\). Then, for any natural number \(K\),
    \[
        \sup_{A} \lvert \mu (A) - \Leb(A)\rvert 
        \leq 10 \Biggl( \frac{1}{K} + \sum_{k=1}^{K} \frac{\bigl\lvert \hat{\mu} (k) \bigr\rvert}{k} \Biggr),
    \]
    where the supremum is over all arcs \(A \subseteq S^1\) of the unit circle, \(\Leb\) is the Lebesgue measure, and \(\hat{\mu}(k)\) is the discrete Fourier transform of $\mu$ evaluated at \(k\).
\end{Theorem}

The following result of Poulias on the number of solutions to a Diophantine inequality will be used in our estimates.

\begin{Theorem}[Poulias \protect{\cite[special case of Theorem 1.1]{poulias2021}}]\label{thm:poulias}
    If $c > 2$ and 
    \[
        2k > \bigl(\floor{2c} + 1 \bigr) \bigl(\floor{2c} + 2 \bigr) + 1,
    \]
    then the number of solutions to 
    \[
        \lvert m_1^c + \cdots + m_k^c - m_{k + 1}^c - \cdots - m_{2k}^c \rvert \leq k
    \]
    in the interval $[1, M]^{2k}$ is $\lesssim M^{2k - c}$, where the implied constant depends on $c$ and $k$, but not on $M$. 
\end{Theorem}

Finally, we quote the following result, which is a classical application of van der Corput's method of bounding exponential sums.

\begin{Theorem}[Graham-Kolesnik \protect{\cite[Theorem 2.9]{grahamKolesnik1991}}]\label{thm:vdc-estimate}
    Let $q$ be a non-negative integer and $f$ be a function with $q + 2$ continuous derivatives on the interval $I \subseteq [X, 2X]$. Assume that there are constants $B_1, B_2$, and $F$ such that 
    \begin{equation}\label{eqn:derv-bound}
        B_1 F X^{-r} \leq \lvert f^{(r)}(x)\rvert \leq B_2 F X^{-r}
    \end{equation} 
    for all $1 \leq r \leq q + 2$. Then there is a constant $B = B(B_1,B_2)$ such that 
    \[
        \biggl\lvert \sum_{n \in I} \e\bigl(f(n) \bigr) \biggr\rvert 
        \leq B \biggl(F^\frac{1}{2^{q+2}-2}X^{1-\frac{q+2}{2^{q+2}-2}} + \frac{X}{F} \biggr).
    \]
\end{Theorem}

\subsection{Outline of Argument}\label{outline}

\hspace{20pt} Our argument is Fourier-analytic in nature. Initially we recast the averages \eqref{eqn:avg} in terms of Fourier transforms. Consequently, the main term (which equals $1$) comes from the trivial character, while the rest of the characters are divided into two disjoint parts. In the first part (treating the case where the character is ``close" to one), we use the combination of an elementary argument using the Prime Number Theorem and Vinogradov's bounds for $\hat{\Lambda}$, depending on whether the character is ``too close" to one or not, and Poulias's bound for the number of solutions to a Diophantine inequality for the other term involving an exponential sum. In the second part, we estimate $\hat{\Lambda}$ trivially and use the \Erdos-\Turan{} inequality to convert the exponential sum involving $\floor{m^c}$ to one involving $m^c$, then use van der Corput estimates for this new sum, which can be applied because of the smoothness of the function $m^c$. Putting everything together, this proves the result. Below, we provide a dependency graph to illustrate how different results proved in this paper are interrelated.

\begin{figure}[H]\label{fig:dependency-graph}
    \centering
    \begin{tikzpicture}[
            every node/.style = {
                draw, rectangle, 
                minimum height = 1.5em, 
                outer sep = 2.5pt, 
                inner ysep = 0pt
            }
        ]
            \node (a) at (8.8em, 0em) {\small{\cref{cor:infinitely-many-pairs}}};
            \node (b) at (8.8em, 3em) {\small{\cref{thm:main-thm}}};
            
            \node (c1) at (5.1em, 6em) {\small{\cref{thm:c-greater-than-2}}};
            \node (c2) at (12.5em, 6em) {\small{\cref{prop:fourier-transform}}};
            
            \node (d1) at (-0.4em, 9em) {\small{\cref{prop:estimate-sigma1}}};
            \node (d2) at (11.1em, 9em) {\small{\cref{prop:estimate-sigma2}}};
            
            \node (e1) at (-13em, 12em) {\small{\cref{lem:parseval-bound}}};
            \node (e2) at (-6.7em, 12em) {\small{\cref{lem:rational-approx-bound}}};
            \node (e3) at (-0.4em, 12em) {\small{\cref{lem:small-xi-bound}}};
            \node (e4) at (5.9em, 12em) {\small{\cref{lem:poulias-bound}}};
    
            \node (f1) at (-6.7em, 15em) {\small{\cref{thm:vinogradov}}};
            \node (f2) at (3.5em, 15em) {\small{\cref{thm:poulias}}};
            \node (f3) at (11.1em, 15em) {\small{\cref{prop:estimate-sigma2-details}}};
    
            \node (g1) at (3.9em, 18em) {\small{\cref{lem:turan-bound}}};
            \node (g2) at (11.1em, 18em) {\small{\cref{cor:bound-for-turan-bound}}};
            \node (g3) at (18.3em, 18em) {\small{\cref{thm:vdc-estimate}}};
    
            \node (h1) at (-6.2em, 21em) {\small{\cref{lem:partition-bound}}};
            \node (h2) at (0.5em, 21em) {\small{\cref{lem:ft-bounds-for-measures}}};
            \node (h3) at (7.2em, 21em) {\small{\cref{lem:expectation-to-measure}}};
            \node (h4) at (14.1em, 21em) {\small{\cref{thm:erdos-turan}}};
    
            \node (i) at (3.9em, 24em) {\small{\cref{lem:measure-stuff}}};

            \graph[use existing nodes]{
                i -> {h2, h3};
                {h1, h2, h3, h4} -> g1;
                {g1, g2, g3} -> f3 ->["\small{\cref{lem:sigma2-proof-works}}"{draw = none}] d2;
                f1 -> e2;
                f2 -> e4;
                {e1, e2, e3, e4} -> d1;
                {d1, d2} -> c1;
                {c1, c2} -> b -> a;
            };
    \end{tikzpicture}
    \caption{Dependency graph. An arrow from vertex \(a\) to vertex \(b\) denotes that the proof of \(b\) depends on \(a\).}
\end{figure}
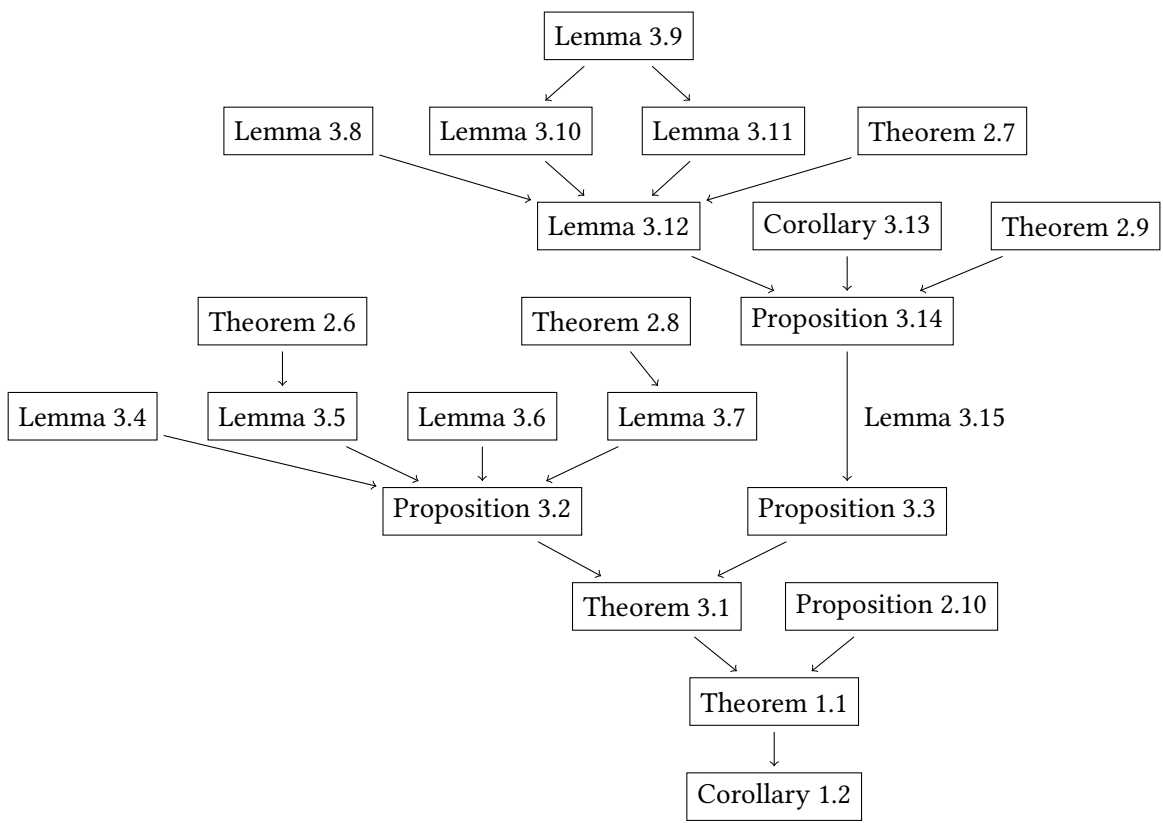

\subsection{Initial Manipulations}
We aim to understand the asymptotic behavior of \(\mathop{\E}_{m \leq M} \mathop{\E}_{n \leq N} \Lambda(n) \Lambda\bigl(n + \floor{m^c} \bigr)\) through the use of the discrete Fourier transform. 

\begin{Proposition}\label{prop:fourier-transform} 
Let $N\in \N$, $c > 2$ be a non-integer, $B \geq \frac{3}{c}$, and $M = \frac{N^{1/c}}{\log^{B} N}$. Then
    \[
        \mathop{\E}_{m \leq M} \mathop{\E}_{n \leq N} \Lambda(n) \Lambda\bigl(n + \floor{m^c} \bigr) 
        = \frac{1}{N^2} \sum_{\xi = 0}^{N - 1} \bigl\lvert \hat{\Lambda}(\xi) \bigr\rvert^2 \mathop{\E}_{m \leq M} \e\biggl(\! -\frac{\xi}{N} \floor{m^c} \!\biggr) + \Oh \biggl(\frac{1}{\log^{Bc-2}{N}}\biggr).
    \]
\end{Proposition}
\begin{proof} First, unpacking notation yields
    \begin{equation}\label{eqn:just-by-definition}
        \mathop{\E}_{m \leq M} \mathop{\E}_{n \leq N} \Lambda(n) \Lambda\bigl(n + \floor{m^c} \bigr)
        = \frac{1}{NM} \sum_{m = 0}^{M - 1} \sum_{n = 0}^{N - 1} \Lambda(n) \Lambda\bigl(n + \floor{m^c} \bigr).
    \end{equation}
        We next apply the Fourier inversion formula (\cref{thm:fourier}) to each of the indicator functions in the above right-hand side. To allow us to do so, we shift the domains of \(\Lambda(n)\) and \(\Lambda\bigl(n + \floor{m^c} \bigr)\) to \(\Z/N\Z\) by taking the arguments modulo \(N\):
    \begin{equation*}
        \eqref{eqn:just-by-definition} = \frac{1}{NM} \sum_{m = 0}^{M - 1} \sum_{n = 0}^{N - 1} \Lambda\bigl(n \text{ (mod } N) \bigr) \Lambda\bigl(n + \floor{m^c} \text{ (mod } N)\bigr) + \Oh \biggl(\frac{M^c}{N} \log^2{N} \biggr).
    \end{equation*}
    This transformation introduces an error of \(\Oh \bigl((M^c/N)\log^2{N} \bigr)\) because the difference between \(\Lambda\bigl(n + \floor{m^c} \bigr)\) and \(\Lambda\bigl(n + \floor{m^c} \text{ (mod } N) \bigr)\) is $\lesssim \log N$, and $\Lambda\bigl(n\text{ (mod } N) \bigr) = \Lambda(n) \lesssim \log N$.
    Now,
    \begin{align*}
        \begin{split}
            \eqref{eqn:just-by-definition} &= \frac{1}{NM} \sum_{m = 0}^{M - 1} \sum_{n = 0}^{N - 1} \Biggl(\frac{1}{N} \sum_{\eta = 0}^{N - 1} \hat{\Lambda}(\eta) \e\biggl(\! -\frac{\eta}{N} n \biggr) \Biggr)\\
            &\hphantom{= \frac{1}{NM} \sum_{m = 0}^{M - 1} \sum_{n = 0}^{N - 1}}\ 
            \cdot \Biggl(\frac{1}{N} \sum_{\xi = 0}^{N - 1} \hat{\Lambda}(\xi) \e\biggl(\! -\frac{\xi}{N} \bigl(n + \floor{m^c} \bigr) \biggr) \Biggr) + \Oh \biggl(\frac{M^c}{N}\log^2{N} \biggr)
        \end{split}\\
        \begin{split}
            &= \frac{1}{N^2M} \sum_{m = 0}^{M - 1} \, \sum_{\eta,\, \xi = 0}^{N - 1} \hat{\Lambda}(\xi)\hat{\Lambda}(\eta) \e\biggl(\! -\frac{\xi}{N} \floor{m^c} \!\biggr)\\
            &\hphantom{= \frac{1}{NM} \sum_{m = 0}^{M - 1} \, \sum_{n = 0}^{N - 1}}\ 
            \cdot \Biggl(\frac{1}{N} \sum_{n = 0}^{N - 1} \e\biggl(\! -\frac{\xi + \eta}{N} n \biggr)\Biggr) + \Oh \biggl(\frac{M^c}{N}\log^2{N} \biggr).
        \end{split}
    \end{align*}
    Note that \(\sum_{n = 0}^{N - 1} \e\bigl(-n(\xi + \eta)/N \bigr) = 0\) unless \(\eta +\xi = N\). So, all of the cross terms cancel, and we are left with the simpler sum
    \begin{align*}
        \mathop{\E}_{m \leq M} \mathop{\E}_{n \leq N} \Lambda(n) \Lambda\bigl(n + \floor{m^c} \!\bigr)
        &= \frac{1}{N^2M} \sum_{m = 0}^{M - 1} \sum_{\xi = 0}^{N - 1} \bigl\lvert \hat{\Lambda}(\xi) \bigr\rvert^2 \e\biggl(\! -\frac{\xi}{N} \floor{m^c} \!\biggr) + \Oh \biggl(\frac{M^c}{N}\log^2{N} \biggr)\\
        &=\frac{1}{N^2} \sum_{\xi = 0}^{N - 1} \bigl\lvert \hat{\Lambda}(\xi) \bigr\rvert^2 \mathop{\E}_{m \leq M} \e\biggl(\! -\frac{\xi}{N} \floor{m^c} \!\biggr) + \Oh \biggl(\frac{1}{\log^{Bc-2}{N}}\biggr),
    \end{align*} where in the first step we used that $\hat{\Lambda}(\xi) = \overline{\hat{\Lambda}(-\xi)}$.
\end{proof}

\section{Proof of Main Result}

In this section we prove \cref{thm:main-thm}, following the outline in \cref{outline}. In view of \cref{prop:fourier-transform}, the following directly implies \cref{thm:main-thm}.

\begin{Theorem}\label{thm:c-greater-than-2}
    Let $c > 2$ be a non-integer.
    For any $A \geq 1$ there exists $B \geq 1$ such that if $M = \frac{N^{1/c}}{\log^{B} N}$, then
    \[
         \sum_{\xi = 0}^{N - 1} \biggl\lvert \frac{\hat{\Lambda}(\xi)}{N} \biggr\rvert^2 \Biggl(\frac{1}{M} \sum_{m = 1}^M \e\biggl(\! -\frac{\xi \floor{m^c}}{N} \biggr)\Biggr) 
         = 1 + \Oh\biggl(\frac{1}{\log^A N} \biggr).
    \]
\end{Theorem}

As mentioned in \cref{outline}, we split this sum into three parts: the main term $\xi = 0$, the characters ``close" to the trivial character, $\Sigma_1$, and the characters ``far" from the trivial character, $\Sigma_2$.

Let $b = 4kA + 2Bc + 4k$, where $k = \ceil*{\bigl((\floor{2c} + 1)(\floor{2c} + 2) + 1 \bigr)/2}$, as in \cref{thm:poulias}. We can write
\begin{equation}\label{eqn:main-sum-decomp}
    \sum_{\xi = 0}^{N - 1} \biggl\lvert \frac{\hat{\Lambda}(\xi)}{N} \biggr\rvert^2 \Biggl(\frac{1}{M} \sum_{m = 1}^M \e\biggl(\! -\frac{\xi \floor{m^c}}{N} \biggr)\Biggr) 
    = \biggl(\frac{1}{N} \sum_{n = 1}^N \Lambda(n) \biggr)^2 + \Sigma_1 + \Sigma_2,
\end{equation}
where
\begin{align*}
    \Sigma_1 &= \sum_{\xi = 1}^{\frac{N}{\log^b N}} \biggl\lvert \frac{\hat{\Lambda}(\xi)}{N} \biggr\rvert^2 \Biggl(\frac{1}{M} \sum_{m=1}^M \e\biggl(\! -\frac{\xi \floor{m^c}}{N} \biggr)\Biggr)\\
\intertext{and} 
    \Sigma_2 &= \mkern-10mu \sum_{\xi = \frac{N}{\log^b N}}^{\frac{N}{2}} \!\! \biggl\lvert \frac{\hat{\Lambda}(\xi)}{N} \biggr\rvert^2 \Biggl(\frac{1}{M} \sum_{m=1}^M \e\biggl(\! -\frac{\xi \floor{m^c}}{N} \biggr)\Biggr).
\end{align*}

Since $\lvert \hat{\Lambda}(\xi) \rvert = \lvert \hat{\Lambda}(-\xi) \rvert$, we can reduce to the case where $\xi \leq N/2$. Furthermore, since $\frac{1}{N} \sum_{n = 1}^{N} \Lambda(n) = 1 + \Oh \bigl(e^{-t \sqrt{\log N}} \bigr)$ (\cref{thm:prime-number-equiv}) and $e^{-t \sqrt{\log N}}$ decays faster than $\log^{-k} N$ for any $k$, it suffices to show the following propositions by \cref{eqn:main-sum-decomp}.

\begin{Proposition}\label{prop:estimate-sigma1}
    For any \(A \geq 1\),
    \[
        \Sigma_1 = \Oh\biggl(\frac{1}{\log^A N} \biggr).
    \]
\end{Proposition}

\begin{Proposition}\label{prop:estimate-sigma2}
    For any \(A \geq 1\),
    \[
        \Sigma_2 = \Oh\biggl(\frac{1}{\log^A N} \biggr).
    \]
\end{Proposition}

Combining \cref{prop:estimate-sigma1} and \cref{prop:estimate-sigma2} gives \cref{thm:main-thm}. We now aim to prove these results in the following two sections.

\subsection{Estimating First Sum}
In this section we prove \cref{prop:estimate-sigma1}. The idea is as follows (it might help to read this paragraph in conjunction with the proof of \cref{prop:estimate-sigma1} at the end of the section): Parseval's identity provides good bounds for the sum of $\lvert \hat{\Lambda} \rvert^2$, and the even power sums of the sum of  $\e(-\xi \floor{m^c}/N)$ have an obvious interpretation as the number of solutions to a Diophantine equation. So, we decouple the $\hat{\Lambda}$ sum from the $\e(-\xi \floor{m^c}/N)$ sum to deal with them separately. In order to decouple, we need to apply H\"{o}lder's inequality, but in doing so the exponent of 2 on $\hat{\Lambda}$ changes. To account for this, we rewrite the exponent as $2 = 1/k + (2 - 1/k)$ and factor out the maximum of the first summand. This leaves us to deal with the maximum value of $\hat{\Lambda}$ and bound the number of solutions to a Diophantine equation.

First we dispose of the trivial term, the second moment of $\hat{\Lambda}/N$.
\begin{Lemma}\label{lem:parseval-bound}
    For any $b \geq 1$ and any $k \geq 1$ the following holds:
    \[
        \Biggl(\sum_{\xi = 1}^{\log^b N} \biggl\lvert \frac{\hat{\Lambda}(\xi)}{N} \biggr\rvert^{2} \Biggr)^{1-\frac{1}{2k}} \mkern-32mu = \Oh(\log N).
    \]
\end{Lemma}
\begin{proof}
    We have
    \begin{equation*}
        \sum_{\xi = 1}^{\frac{N}{\log^b N}} \biggl\lvert \frac{\hat{\Lambda}(\xi)}{N} \biggr\rvert^{2}
        \leq \sum_{\xi = 1}^{N} \biggl\lvert \frac{\hat{\Lambda}(\xi)}{N} \biggr\rvert^{2}
        \! = \frac{1}{N}\sum_{\xi = 1}^{N} \bigl(\Lambda(\xi)\bigr)^2,
    \end{equation*}
    where the equality holds by Parseval's identity (\cref{thm:parseval}). Now,
    \begin{equation*}
        \frac{1}{N} \sum_{\xi = 1}^{N} \bigl(\Lambda(\xi)\bigr)^2
        \leq \frac{1}{N} \biggl(\sum_{\xi = 1}^{N} \Lambda(\xi) \biggr) \max_{\xi} \Lambda(\xi)
        \leq \frac{1}{N} \biggl(\sum_{\xi = 1}^{N} \Lambda(\xi) \biggr) \log N.
    \end{equation*}
   By the Prime Number Theorem, $\displaystyle \lim_{N \to \infty} \frac{1}{N} \biggl(\sum_{\xi = 1}^{N} \Lambda(\xi) \biggr) = 1$ and the result follows.
\end{proof}

Now we deal with the supremum of $\hat{\Lambda}(\xi)$ over $\xi \leq N/\log^b N$. When $\xi \geq \log^b N$, we can use Vinogradov's bounds because $\xi/N$ has a rational approximation with a denominator that is not too small (recall that when the denominator is too small, Vinogradov's theorem does not give a useful estimate). In the other case, we use a more elementary argument, approximating $\e(\cdot)$ by simple functions and using the Prime Number Theorem. \cref{lem:rational-approx-bound} addresses the first case and \cref{lem:small-xi-bound} deals with the second. Together, they bound $\bigl\lvert \hat{\Lambda}(\xi) \bigr\rvert$ uniformly over $\xi \leq N/\log^b N$.

\begin{Lemma}\label{lem:rational-approx-bound}
    Suppose $b\geq 16$. The following holds:
    \[
        \sup_{\mathclap{\log^b N < \xi \leq \frac{N}{\log^b N}}}\; \bigl\lvert \hat{\Lambda}(\xi) \bigr\rvert
        = \Oh\biggl(\frac{N}{\log^{\frac{b}{4}} N} \biggr).
    \]
\end{Lemma}
\begin{proof}
    First, observe that 
    $\displaystyle \lim_{N \to \infty} \biggl(\frac{1}{\log^{\frac{b}{2}} N} + \frac{\log^b N}{N} \biggr) = 0$. 
    Therefore, there exists $N_0 \in \N$ such that, for all \(N \geq N_0\),
    \[
        \frac{N}{\log^{\frac{b}{2}} N} + \log^b N < N.
    \]
    
    Suppose that $N \geq N_0$ and $\xi \in (\log^b N, N/\log^b N]$. By Dirichlet's approximation theorem, we can find $a \in \Z$ and $q \leq N/\log^b N$ such that $\gcd(a, q) = 1$ and 
    \[
        \biggl\lvert\frac{\xi}{N} - \frac{a}{q} \biggr\rvert 
        \leq \frac{1}{\frac{N}{\log^b N} q} \leq \frac{1}{q^2}.
    \]
    Clearing denominators, the first inequality gives us \(\lvert q\xi - aN \rvert \leq \log^b N\), meaning \(q\xi \geq aN - \log^b N\). If \(a = 0\), then $\lvert q\xi \rvert \leq \log^b N$. However, by our choice of \(\xi\), $\log^b N < q\xi$, a contradiction. Thus, \(a \neq 0\), and
    \begin{equation*}
         q\xi \geq aN - \log^b N \geq N - \log^b N.
    \end{equation*}
    
    Now, aiming for another contradiction, assume that $q \leq \log^{b/2} N$. As such, 
    \begin{equation*}
        q\xi \leq (\log^{\frac{b}{2}} N) \frac{N}{\log^b N}
        = \frac{N}{\log^{\frac{b}{2}} N}.
    \end{equation*}
    Rearranging terms gives us that
    \begin{equation*}
        N - \log^b N \leq \frac{N}{\log^{\frac{b}{2}} N}.
    \end{equation*}
    However, by hypothesis, \(N > (N/\log^{b/2} N) + \log^b N\), which yields a contradiction. Therefore, $q > \log^{b/2} N$.
    
    Next, applying Vinogradov's theorem (\cref{thm:vinogradov}) gives us
    \begin{align*}
        \biggl\lvert \sum_{n = 0}^{N - 1} \Lambda(n) \e\biggl(\frac{\xi}{N}n \biggr) \biggr\rvert
        &\leq C \biggl(\frac{N}{\sqrt{q}} + N^{\frac{4}{5}} + \sqrt{qN} \biggr) \log^4 N\\
        &\leq C \biggl(\frac{N}{\log^{\frac{b}{4}} N} + N^{\frac{4}{5}} + \frac{N}{\log^{\frac{b}{2}} N} \biggr) \log^4 N.\\
    \intertext{Thus, since the above holds for all $\xi \in (\log^b N, N/\log^b N]$,}
        \sup_{\mathclap{\log^b N < \xi \leq \frac{N}{\log^b N}}}\; \bigl\lvert \hat{\Lambda}(\xi) \bigr\rvert 
        &\leq C \biggl(\frac{N}{\log^{\frac{b}{4}} N} + N^{\frac{4}{5}} +\frac{N}{\log^{\frac{b}{2} - 4} N}\biggr)\\
    \intertext{for all $N \geq N_0$. Therefore,}
         \sup_{\mathclap{\log^b N < \xi \leq \frac{N}{\log^b N}}}\; \bigl\lvert \hat{\Lambda}(\xi) \bigr\rvert
         &= \Oh\biggl(\frac{N}{\log^{\frac{b}{4}} N} \biggr),
    \end{align*}
    since, if $b \geq 16$, then $b/2 - 4 \geq b/4$.
\end{proof}

\begin{Lemma}\label{lem:small-xi-bound}
    For any $b, m \geq 1$,
    \[
        \sup_{1 \leq \xi \leq \log^b N}  \biggl\lvert \frac{1}{N} \sum_{n = 0}^{N - 1} \Lambda(n) \e\biggl(\frac{\xi}{N} n \biggr)\biggr\rvert 
        = \Oh\biggl(\frac{1}{\log^m N} \biggr).
    \]
\end{Lemma}
\begin{proof}
    By the Prime Number Theorem (\cref{thm:prime-number-equiv}), there exist $N_1 \in \N,  C, t > 0$ such that 
    \begin{equation*}
        \biggl\lvert \frac{1}{N} \sum_{n = 0}^{N - 1} \Lambda(n) - 1 \biggr\rvert 
        \leq \frac{C}{e^{t\sqrt{\log N}}} 
    \end{equation*}
    for all \(N \geq N_1\).
    Now, since $\displaystyle \lim_{N \to \infty} \frac{N}{\log^{m+b} N} = \infty$, there exists $N_2 \in \N$ so that, for all \(N \geq N_2\),
    \[
        \frac{N}{\log^{m + b} N} \geq N_1.
    \]
    Further, it can be verified that 
    \begin{equation*}
        \lim_{N \to \infty} \frac{2C\log^{m + b} N}{e^{t \sqrt{\log \frac{N}{\log^{m + b} N}}}} = 0,
    \end{equation*}
    so we can find $N_3 \in \N$ such that, for all \(N \geq N_3\),
    \begin{equation*}
        \frac{2C\log^{m + b} N}{e^{t \sqrt{\log \frac{N}{\log^{m + b} N}}}}
        \leq \frac{1}{\log^m N}.
    \end{equation*}

    With the above in mind, suppose $N \geq \max\{N_2, N_3\}$ and let $k = \floor{\log^m N}$. Accordingly,
    \begin{align}
        \biggl\lvert \frac{1}{N} \sum_{n = 0}^{N - 1} \Lambda(n) \e\Bigl(\frac{\xi}{N} n \Bigr) \biggr\rvert
        &\leq \biggl\lvert \frac{1}{N} \sum_{n = 0}^{N - 1} \Lambda(n) f_k\Bigl(\frac{\xi}{N} n \Bigr) \biggr\rvert 
        + \biggl\lvert \frac{1}{N} \sum_{n = 0}^{N - 1} \Lambda(n) \biggl(\e\Bigl(\frac{\xi}{N} n \Bigr) - f_k\Bigl(\frac{\xi}{N} n \Bigr)\biggr)\biggr\rvert \nonumber\\
        &\leq \biggl\lvert \frac{1}{N} \sum_{n = 0}^{N - 1} \Lambda(n) f_k\Bigl(\frac{\xi}{N} n \Bigr) \biggr\rvert + \frac{2\pi}{kN} \sum_{n = 0}^{N - 1} \Lambda(n),\label{eqn:from-N-big}
    \end{align}
    where $f_k \colon \R \to \C$ is the function satisfying $f_k(x) = \e\bigl(\floor{kx}/k \bigr)$.
    
    \cref{eqn:from-N-big} holds for all $\xi \in [0, N - 1]$. Suppose \(r \in \N\) such that $1 \leq r \leq \log^b N$. Then,
    \begin{equation}\label{eqn:restricting-xi}
        \biggl\lvert \frac{1}{N} \sum_{n = 0}^{N - 1} \Lambda(n) f_k\Bigl(\frac{rn}{N} \Bigr) \biggr\rvert 
        = \biggl\lvert \frac{1}{N} \sum_{\xi = 0}^{rk - 1} \e\biggl(\frac{\xi}{k} \biggr) \; \cdot \;
        \smashoperator{\sum_{\mkern26mu \frac{\xi N}{rk} \leq n < \frac{(\xi + 1) N}{rk}}}  \Lambda(n) \biggr\rvert.
    \end{equation}
    Further, for each $\xi = 0, 1, \ldots, rk - 1$,
    \begin{align*}
        \biggl\lvert \frac{1}{N} \; \cdot \; \smashoperator{\sum_{\mkern26mu \frac{\xi N}{rk} \leq n < \frac{(\xi + 1)N}{rk}}} \Lambda(n) - \frac{1}{rk} \biggr\rvert
        &= \Biggl\lvert \biggl(\frac{\frac{\xi + 1}{rk}}{\frac{(\xi + 1)N}{rk}} \:\smashoperator{\sum_{\mkern14mu n < \frac{(\xi + 1)N}{{rk}}}} \Lambda(n) - \frac{\xi + 1}{rk} \biggr) 
        - \biggl(\frac{\frac{\xi}{rk}}{\frac{\xi N}{rk}} \:\smashoperator{\sum_{\, n < \frac{\xi N}{rk}}} \Lambda(n) - \frac{\xi}{rk} \biggr)\Biggr\rvert\\
        &\leq \frac{\xi + 1}{rk} \Biggl\lvert \frac{1}{\frac{(\xi + 1)N}{rk}} \:\smashoperator{\sum_{\mkern14mu n < \frac{(\xi + 1)N}{{rk}}}} \Lambda(n) - 1 \Biggr\rvert 
        + \frac{\xi}{rk} \Biggl\lvert \frac{1}{\frac{\xi N}{rk}} \:\smashoperator{\sum_{\, n < \frac{\xi N}{rk}}} \Lambda(n) - 1 \Biggr\rvert\\
        &\leq \frac{2C}{e^{t\sqrt{\log \frac{\xi N}{rk}}}},
    \end{align*}
    because if $\xi \geq 1$, then
    \begin{equation*}
        \frac{\xi N}{rk} \geq \frac{N}{\log^{m + b} N}\geq N_1.
    \end{equation*}
    
    Denote $\displaystyle a(\xi) = \frac{1}{N} \;\cdot\; \smashoperator{\sum_{\mkern26mu \frac{\xi N}{rk} \leq n < \frac{(\xi + 1)N}{rk}}} \Lambda(n) - \frac{1}{rk}$. We have
    \begin{align}
        \biggl\lvert \frac{1}{N} \sum_{\xi = 0}^{rk - 1} \e\biggl(\frac{\xi}{k} \biggr) 
        \;\cdot\;
        \smashoperator{\sum_{\mkern26mu \frac{\xi N}{rk} \leq n < \frac{(\xi + 1)N}{rk}}} \Lambda(n) \biggr\rvert
        &= \Biggl\lvert \sum_{\xi = 0}^{rk - 1} \e\biggl(\frac{\xi}{k} \biggr) \biggl(\frac{1}{rk} + a(\xi) \biggr) \Biggr\rvert \nonumber\\
        &= \Biggl\lvert \sum_{\xi = 0}^{rk - 1} \e\biggl(\frac{\xi}{k} \biggr) a(\xi) \Biggr\rvert \nonumber\\
        &\leq \frac{2Crk}{e^{t\sqrt{\log \frac{N}{rk}}}} \nonumber\\ 
        &\leq \frac{2C\log^{m + b} N}{e^{t \sqrt{\log \frac{N}{\log^{m + b} N}}}} \nonumber\\
        &\leq \frac{1}{\log^m N},\label{eqn:bounding-restricted-xi}
    \end{align}
    where the final inequality holds because $N \geq N_3$.
    
    By \cref{eqn:from-N-big}, \cref{eqn:restricting-xi}, and \cref{eqn:bounding-restricted-xi}, we obtain that, for any $\xi \in [1, \log^b N]$,
    \begin{align*}
        \biggl\lvert \frac{1}{N} \sum_{n = 0}^{N - 1} \Lambda(n) \e\biggl(\frac{\xi}{N} n \biggr) \biggr\rvert
        &\leq \frac{1}{\log^m N} + \frac{2\pi}{kN} \sum_{n = 0}^{N - 1} \Lambda(n).\\
    \intertext{Therefore, for all \(N \geq \max\{N_2, N_3\}\),}
        \sup_{1 \leq \xi \leq \log^b N} \biggl\lvert \frac{1}{N} \sum_{n = 0}^{N - 1} \Lambda(n) \e\biggl(\frac{\xi}{N} n \biggr) \biggr\rvert
        &\leq \frac{1}{\log^m N} + \frac{2\pi}{kN} \sum_{n = 0}^{N - 1} \Lambda(n).\\
    \intertext{So,}
        \sup_{1\leq \xi \leq \log^b N} \biggl\lvert \frac{1}{N} \sum_{n = 0}^{N - 1} \Lambda(n) \e\biggl(\frac{\xi}{N} n \biggr) \biggr\rvert 
        &= \Oh \biggl(\frac{1}{\log^m N} \biggr),
    \end{align*} 
    as desired.
\end{proof}

Now, we only need to bound the $2k$-th power sum of the absolute value of the sum involving $\e(-\xi \floor{m^c}/N)$. As mentioned, this is equivalent to counting the number of solutions to a Diophantine equation. Even though the Diophantine equation \eqref{eqn:2k-dim-cube} turns out to be one that is not easy to solve, we can bound the number of its solutions by the number of integer solutions to an inequality not involving the floor function \eqref{eqn:diophantine-inequality} by using the triangle inequality. This inequality has been studied before, so we invoke Poulias's result \cite{poulias2021} to bound it. Some calculation yields the desired result.

\begin{Lemma}\label{lem:poulias-bound}
    Suppose $b, c \in \R$ and \(k \in \N\), with \(b \geq 1\), \(c > 2\), and $2k > (\floor{2c} + 1)(\floor{2c} + 2) + 1$. If $M = \frac{N^{1/c}}{\log^B N}$ for some $B \geq 1$, then
    \[
        \sum_{\xi = 1}^{\frac{N}{\log^b N}} \biggl\lvert \mathop{\E}_{m \leq M} \e\biggl(\!-\frac{\xi \floor{m^c}}{N}  \biggr) \biggr\rvert^{2k}
        \!\!\! = \Oh\bigl(\log^{Bc} N \bigr).
    \]
\end{Lemma}
\begin{proof}
    Suppose $N$ is large (which we will contextualize later). The following holds:
    \begin{align*}
        \sum_{\xi = 1}^{\frac{N}{\log^b N}} \biggl\lvert \mathop{\E}_{m \leq M} \e\biggl(\! -\frac{\xi \floor{m^c}}{N} \biggr) \biggr\rvert^{2k}
        \!\! &\leq \sum_{\xi = 1}^{N}\: \biggl\lvert \mathop{\E}_{m \leq M} \e\biggl(\! -\frac{\xi \floor{m^c}}{N} \biggr) \biggr\rvert^{2k}\\
        &= \sum_{\xi = 0}^{N - 1} \biggl(\mathop{\E}_{m \leq M} \e\biggl(\! -\frac{\xi \floor{m^c}}{N} \biggr)\biggr)^k \biggl(\overline{\mathop{\E}_{m \leq M} \e\biggl(\! -\frac{\xi \floor{m^c}}{N} \biggr)}\biggr)^k\\
        \begin{split}
            &= \sum_{\xi = 0}^{N - 1} \Biggl(\frac{1}{M^k} \;\cdot\; 
            \smashoperator{\sum_{\mkern30mu m_1, \ldots, m_k \leq M}}\quad\, \e\biggl(\! -\frac{\xi}{N} \bigl(\floor{m_1^c} + \cdots + \floor{m_k^c} \bigr)\biggr)\Biggr)\\
            &\hphantom{= \sum_{\xi = 0}^{N - 1}}\qquad\: \cdot 
            \Biggl(\frac{1}{M^k} \;\cdot\; \smashoperator{\sum_{\mkern20mu m_{k + 1}, \ldots, m_{2k} \leq M}}\quad\;\, \e\biggl(\frac{\xi}{N} \bigl(\floor{m_{k+1}^c} + \cdots + \floor{m_{2k}^c} \bigr)\biggr)\Biggr)
        \end{split}\\
        \begin{split}
            &= \frac{1}{M^{2k}} \! \sum_{\xi = 0}^{N - 1} \mkern-25mu \smashoperator[r]{\sum_{\mkern36mu m_1, \ldots, m_{2k} \leq M}}\quad\;\, \e\biggl(\! -\frac{\xi}{N} \bigl(\floor{m_1^c} + \cdots + \floor{m_k^c}\\[-1.2em]
            &\hphantom{= \frac{1}{M^{2k}} \! \sum_{\xi = 0}^{N - 1} \mkern-25mu \smashoperator[r]{\sum_{\mkern36mu m_1, \ldots, m_{2k} \leq M}}\quad\;\, \e\biggl(\! -\frac{\xi}{N} \bigl(\floor{m_1^c}}\:
            - \floor{m_{k + 1}^c} - \cdots - \floor{m_{2k}^c} \bigr)\biggr).
        \end{split}
    \end{align*}
    
    Note that, for any $R \in \Z$,
    \begin{equation*}
      \sum_{\xi = 0}^{N - 1} \e\biggl(\! -\frac{\xi}{N} R \biggr) 
      =\begin{cases}
            N &\text{if } N \mid R\\
            \hfil 0 &\text{if } N \nmid R
        \end{cases}.
    \end{equation*}
    
    By hypothesis, we take \(N\) sufficiently large to have
    \begin{align*}
        \begin{split}
            \bigl\lvert \floor{m_1^c} + \cdots + \floor{m_k^c} - \floor{m_{k + 1}^c} &- \cdots - \floor{m_{2k}^c} \bigr\rvert\\
            &\leq \bigl\lvert \floor{m_1^c} - \floor{m_{k + 1}^c} \bigr\rvert + \cdots + \bigl\lvert \floor{m_k^c} - \floor{m_{2k}^c} \bigr\rvert
        \end{split}\\
        &\leq k M^c = \frac{kN}{\log^{Bc} N} < N.
    \end{align*}
    Therefore,
    \begin{equation*}
        \sum_{\xi = 1}^{N} \biggl\lvert \mathop{\E}_{m \leq M} \e\biggl(\! -\frac{\xi \floor{m^c}}{N} \biggr) \biggr\rvert^{2k} 
        \!\!\! = \frac{N}{M^{2k}} \cdot G(M),
    \end{equation*}
    where $G(M)$ is the number of solutions to the equation 
    \begin{equation}\label{eqn:2k-dim-cube}
        \floor{m_1^c} + \cdots + \floor{m_k^c} - \floor{m_{k + 1}^c} - \cdots - \floor{m_{2k}^c} = 0
    \end{equation}
    inside the $2k$-dimensional cube $\{1, \ldots , M\}^{2k}$. 

    Now, suppose that $(m_1, \ldots, m_{2k})$ satisfies \cref{eqn:2k-dim-cube}. Then,
    \begin{align}
        \begin{split}\nonumber
            \bigl\lvert m_1^c + \cdots + m_k^c - m_{k + 1}^c &- \cdots - m_{2k}^c \bigr\rvert\\
            &= \bigl\lvert \{m_1^c\} - \{m_{k + 1}^c\} + \cdots + \{m_k^c\} - \{m_{2k}^c\} \bigr\rvert
        \end{split}\\ 
        &\leq \bigl\lvert \{m_1^c\} - \{m_{k + 1}^c\} \bigr\rvert + \cdots + \bigl\lvert \{m_k^c\} - \{m_{2k}^c\} \bigr\rvert < k. \label{eqn:diophantine-inequality}
    \end{align}
    Thus, $G(M)$ is less than or equal to the number of solutions of the above inequality. By Poulias (\cref{thm:poulias}), this number is 
    \[
        2k \Omega M^{2k - c} + \oh(M^{2k - c}) = \Oh(M^{2k - c}),
    \]
    where $\Omega$ is a positive constant. So, $G(M) = \Oh(M^{2k - c})$ and 
    \begin{equation*}
       \sum_{\xi = 1}^{\frac{N}{\log^b N}} \biggl\lvert \mathop{\E}_{m \leq M} \e\biggl(\! -\frac{\xi \floor{m^c}}{N} \biggr) \biggr\rvert^{2k} \!\!\!
       = \Oh \biggl(\frac{N}{M^c} \biggr)
       = \Oh\bigl(\log^{Bc} N \bigr).
    \end{equation*}
\end{proof}

Now we have all the ingredients to prove \cref{prop:estimate-sigma1}. We carry out the procedure mentioned in the beginning of this section.

\begin{proof}[Proof of \cref{prop:estimate-sigma1}]
    Suppose $k$ is as in \cref{lem:poulias-bound}. Observe that
    \begin{align*}
        \sum_{\xi = 1}^{\frac{N}{\log^b N}} \biggl\lvert \frac{\hat{\Lambda}(\xi)}{N} \biggr\rvert^2 
        \!\biggl(\mathop{\E}_{m \leq M} \e\biggl(\! &-\frac{\xi \floor{m^c}}{N} \biggr)\biggr)
        \leq \sum_{\xi = 1}^{\frac{N}{\log^b N}} \biggl\lvert \frac{\hat{\Lambda}(\xi)}{N} \biggr\rvert^{\frac{1}{k}} \biggl\lvert \frac{\hat{\Lambda}(\xi)}{N} \biggr\rvert^{2 - \frac{1}{k}} \biggl\lvert \mathop{\E}_{m \leq M} \e\biggl(\! -\frac{\xi \floor{m^c}}{N} \biggr)\biggr\rvert\\
        &\leq\: \sup_{\xi}\; \biggl\lvert \frac{\hat{\Lambda}(\xi)}{N} \biggr\rvert^{\frac{1}{k}}\sum_{\xi = 1}^{\frac{N}{\log^b N}} \biggl\lvert \frac{\hat{\Lambda}(\xi)}{N} \biggr\rvert^{2 - \frac{1}{k}} \biggl\lvert \mathop{\E}_{m \leq M} \e\biggl(\! -\frac{\xi \floor{m^c}}{N} \biggr)\biggr\rvert\\
        &\leq\: \sup_{\xi}\; \biggl\lvert \frac{\hat{\Lambda}(\xi)}{N} \biggr\rvert^{\frac{1}{k}} \Biggl(\sum_{\xi = 1}^{\frac{N}{\log^b N}} \biggl\lvert \frac{\hat{\Lambda}(\xi)}{N} \biggr\rvert ^2 \Biggr)^{\!\!\! 1 - \frac{1}{2k}}
        \! \Biggl(\sum_{\xi = 1}^{\frac{N}{\log^b N}} \biggl\lvert \mathop{\E}_{m \leq M} \e\biggl(\! -\frac{\xi \floor{m^c}}{N} \biggr)\biggr\rvert^{2k} \Biggr)^{\frac{1}{2k}}
    \end{align*}
    by H\"{o}lder's inequality. Therefore, by the previous lemmas,
    \begin{align*}
        \sum_{\xi = 1}^{\frac{N}{\log^b N}} \biggl\lvert \frac{\hat{\Lambda}(\xi)}{N} \biggr\rvert^2 \Biggl(\frac{1}{M} \sum_{m = 1}^M \e\biggl(\! -\frac{\xi \floor{m^c}}{N} \biggr)\Biggr)
        &= \underbrace{\Oh\biggl(\frac{1}{\log^{A + 1 + \frac{Bc}{2k}} N} \biggr)}_{\cref{lem:rational-approx-bound},\ \cref{lem:small-xi-bound}} \underbrace{\Oh(\log N)}_{\cref{lem:parseval-bound}} \underbrace{\Oh\bigl(\log^{\frac{Bc}{2k}} N \bigr)}_{\cref{lem:poulias-bound}}\\
        &= \Oh\biggl(\frac{1}{\log^A N}\biggr).
    \end{align*}
\end{proof}

\subsection{Estimating Second Sum}

In this section we prove \cref{prop:estimate-sigma2}. The idea is as follows: We factor out the supremum of the exponential sum, which decouples the product and allows us to use Parseval's identity (\cref{thm:parseval}) on the $\hat{\Lambda}$ term. It then remains to bound $\mathop{\E}_{m \leq M} \e(-\xi \floor{m^c}/N)$ uniformly for $N/\log^b N \leq \xi \leq N/2$. For this, we first partition the unit interval into intervals \(P\) of equal length, then approximate the sum of interest by a sum involving $\e(-\xi m^c/N)$. After doing so, the approximating sum still includes a term $\1_{P}\bigl(\{m^c\}\bigr)$ that needs to be dealt with. We do this by writing this sum in terms of probability measures and using the \Erdos-\Turan{} inequality (\cref{thm:erdos-turan}). This process reduces the problem into bounding exponential sums of the form $\mathop{\E}_{m \leq M} \e(um^c)$, which we do by van der Corput's method (\cref{thm:vdc-estimate}). For \(c > 12\), Deshouillers has proven better bounds for these sums \cite{deshouillers1973}, but van der Corput's method applies to all $c > 0$, so we prefer to use this method for a unified treatment. 

In summary, to estimate the behavior of the expression
\[
    \mathop{\E}_{m \leq M} \e\biggl(\! -\frac{\xi}{N} \floor{m^c} \!\biggr) = \mathop{\E}_{m \leq M} \e\biggl(\! -\frac{\xi}{N} m^c \biggr) \e\biggl(\frac{\xi}{N} \{m^c\} \!\biggr),
\] 
we bound it by an expression depending on \(\mathop{\E}_{m \leq M} \e(-\xi m^c/N)\), whose asymptotics are more tractable. However, this new expression contains the problematic term $\1_{P}\bigl(\{m^c\}\bigr)$, which we deal with by applying the \Erdos-\Turan{} inequality.

\begin{Lemma}\label{lem:partition-bound} 
    Let \(\textbf{P}\) be a partition of \(S^1\) into intervals \(P\) of equal length \(\gamma\). Then,
    \[
        \biggl\lvert \mathop{\E}_{m \leq M} \e\biggl(\! -\frac{\xi}{N} \floor{m^c} \!\biggr) \biggr\rvert 
        \leq \sum_{P \in \textbf{P}}\: \biggl\lvert \mathop{\E}_{m \leq M} \1_{P}\bigl(\{m^c\}\bigr) \e\biggl(\! -\frac{\xi}{N} m^c \biggr) \biggr\rvert + 4\pi\frac{\xi}{N} \gamma.
    \]
\end{Lemma}

\begin{proof}
    Let \(\textbf{P}\) be a partition of \(S^1\) into intervals \([a_i, b_i)\) of length \(\gamma\). Then,
    \begin{equation}\label{eqn:partition-setup}
        \biggl\lvert \mathop{\E}_{m \leq M} \e\biggl(\! -\frac{\xi}{N} m^c \biggr) \e\biggl(\frac{\xi}{N} \{m^c\} \!\biggr) \biggr\rvert
        = \biggl\lvert \sum_{[a_i, b_i) \in \textbf{P}} \mathop{\E}_{m \leq M} \1_{[a_i, b_i)}\bigl(\{m^c\}\bigr) \e\biggl(\! -\frac{\xi}{N} m^c \biggr) \e\biggl(\frac{\xi}{N} \{m^c\} \!\biggr) \biggr\rvert.
    \end{equation}
    
    Now, since summands in the above expression are nonzero only when \(\{m_c\} \in [a_i, b_i)\), by the triangle inequality we have
    \begin{align}
        \begin{split}\nonumber
            \eqref{eqn:partition-setup} &\leq \biggl\lvert \sum_{[a_i, b_i) \in \textbf{P}} \mathop{\E}_{m \leq M} \1_{[a_i, b_i)}\bigl(\{m^c\}\bigr) \e\biggl(\! -\frac{\xi}{N} m^c \biggr)
            \underbrace{\e\biggl(\frac{\xi}{N} a_i \biggr)}_{\leq 1} \biggr\rvert
            + \sum_{[a_i, b_i) \in \textbf{P}} \mathop{\E}_{m \leq M}
            \biggl\lvert
            \underbrace{\e\biggl(\! -\frac{\xi}{N} m^c\biggr)}_{\leq 1}\biggr\rvert\\[-0.7em]
            &\qquad\qquad\qquad\; \cdot\; \biggl\lvert \1_{[a_i, b_i)}\bigl(\{m^c\}\bigr) \e\biggl(\frac{\xi}{N} \{m^c\} \!\biggr) - \1_{[a_i, b_i)}\bigl(\{m^c\}\bigr) \e\biggl(\frac{\xi}{N} a_i \biggr) \biggr\rvert
        \end{split}\\[.2em]
        \begin{split}\nonumber
            &\leq \sum_{[a_i, b_i) \in \textbf{P}}\, \biggl\lvert \mathop{\E}_{m \leq M} \1_{[a_i, b_i)}\bigl(\{m^c\}\bigr) \e\biggl(\! -\frac{\xi}{N} m^c \biggr) \biggr\rvert\\[-0.7em]
            &\qquad\qquad\qquad + \underbrace{\biggl(\sum_{[a_i, b_i) \in \textbf{P}} \mathop{\E}_{m \leq M} \1_{[a_i, b_i)}\bigl(\{m^c\}\bigr) \biggr)}_{\leq 1}
            \Biggl(\sup_{\substack{\lvert y - x \rvert \leq \gamma \\ 0 \leq x, y \leq 1}} \biggl\lvert \, \e\biggl(\frac{\xi}{N} y \!\biggr) - \e\biggl(\frac{\xi}{N} x \!\biggr) \biggr\rvert \Biggr)
        \end{split}\\
            &\leq \sum_{[a_i, b_i) \in \textbf{P}}\, \biggl\lvert \mathop{\E}_{m \leq M} \1_{[a_i, b_i)}\bigl(\{m^c\}\bigr) \e\biggl(\! -\frac{\xi}{N} m^c \biggr) \biggr\rvert + 4\pi\frac{\xi}{N} \gamma. \label{eqn:partition-sum}
    \end{align}
\end{proof}

Now, we encode the behavior of \cref{eqn:partition-sum} with a measure $\mu$, which satisfies $\mu(P) = \mathop{\E}_{m \leq M} \1_{P}\bigl(\{m^c\}\bigr) \e(-\xi m^c/N)$, and bound it by the difference between some probability measures and the Lebesgue measure, which are amenable to the \Erdos-\Turan{} inequality.

\begin{Lemma}\label{lem:measure-stuff} 
    Define the measures
    \begin{align*}
        &\mu = \!\! \mathop{\E}_{m \leq M} \e\biggl(\! -\frac{\xi}{N} m^c \biggr) \delta_{\{m^c\}},\\ 
        &\mu_1 = \frac{1}{1 + c_1}\mathop{\E}_{m\leq M} \biggl(1 + \cos\biggl(\! -\frac{2 \pi i \xi m^c}{N} \biggr) \biggr)\delta_{\{m^c\}},\\
        &\mu_2 = \frac{1}{1 + c_2}\mathop{\E}_{m \leq M}\biggl(1 + \sin\biggl(\! -\frac{2 \pi i \xi m^c}{N}\biggr)\biggr)\delta_{\{m^c\}}, \text{ and } \\
        &\mu_3 = \!\! \mathop{\E}_{m \leq M} \delta_{\{m^c\}},
    \end{align*}
    where
    \begin{equation*}
        c_1 = \!\! \mathop{\E}_{m \leq M} \cos\biggl(\! -\frac{2\pi i \xi m^c}{N} \biggr) \text{ and } 
        c_2 = \!\! \mathop{\E}_{m \leq M} \sin\biggl(\! -\frac{2\pi i \xi m^c}{N}\biggr).
    \end{equation*}
    Then
    \begin{equation*}
        \lvert \mu(P) \rvert \leq 5 \Bigl(\bigl\lvert \mu_1(P) - \Leb(P) \bigr\rvert +  \bigl\lvert \mu_2(P) - \Leb(P) \bigr\rvert + \bigl\lvert \mathop{\E}_{m \leq M} \delta_{\{m^c\}}(P) - \Leb(P) \bigr\rvert \Bigr).
    \end{equation*}
\end{Lemma}
Note that $\mu_1$ and $\mu_2$ are probability measures on the unit circle.
\begin{proof}
    Observe that
    \[
        \mu = (1 + c_1)\mu_1 + i(1 + c_2)\mu_2 - \bigl((1 + c_1)\mathop{\E}_{m \leq M} \delta_{\{m^c\}} + i(1 + c_2)\mathop{\E}_{m \leq M}\delta_{\{m^c\}}\bigr).
    \]
    By the triangle inequality,
    \[
        \lvert \mu \rvert \leq (1 + c_1) \bigl(\lvert \mu_1 - \Leb\rvert + \lvert\mathop{\E}_{m \leq M}\delta_{\{m^c\}} - \Leb \rvert \bigr) + (1 + c_2) \bigl(\lvert \mu_2 - \Leb \rvert + \lvert \mathop{\E}_{m \leq M}\delta_{\{m^c\}} - \Leb \rvert \bigr).
    \]
    Noting that $c_1, c_2 \leq 1$, we conclude.
\end{proof}

We will apply the \Erdos-\Turan{} inequality to $\mu_1, \mu_2$, and $\mu_3$. For this, we compute the Fourier coefficients of these measures. As noted in the beginning of the section, we bound the obtained coefficients by sums of the form $\mathop{\E}_{m \leq M} \e(um^c)$.

For $k \geq 1$, in view of the trivial identity $\bigl\{k\{m^c\}\bigr\} = \{km^c\}$, 
\begin{align*}
    \hat{\mu}_1(k) &= \frac{1}{1 + c_1} \biggl(\mathop{\E}_{m \leq M}\e\biggl(\! -\frac{k m^c}{N} \biggr) + \mathop{\E}_{m \leq M} \e\biggl(\! -\frac{k m^c}{N} \biggr) \cos\biggl(\! -\frac{2\pi i \xi m^c}{N} \biggr) \biggr),\\
    \hat{\mu}_2(k) &= \frac{1}{1 + c_2} \biggl(\mathop{\E}_{m \leq M} \e\biggl(\! -\frac{km^c}{N} \biggr) + \mathop{\E}_{m \leq M} \e\biggl(\! -\frac{km^c}{N} \biggr) \sin\biggl(\! -\frac{2\pi i \xi m^c}{N} \biggr) \biggr), \text{ and}\\
    \hat{\mu}_3(k) &= \mathop{\E}_{m \leq M} \e\biggl(\! -\frac{km^c}{N} \biggr). 
\end{align*} 

\begin{Lemma}\label{lem:ft-bounds-for-measures}
    For $k \geq 1$,
    \begin{equation*}
        \bigl\lvert \hat{\mu}_1(k) \bigr\rvert, \bigl\lvert \hat{\mu}_2(k) \bigr\rvert 
        \leq 2\Biggl(\biggl\lvert \mathop{\E}_{m \leq M} \e\biggl(\Bigl(k - \frac{\xi}{N} \Bigr) m^c \biggr) \biggr\rvert 
        + \biggl\lvert \mathop{\E}_{m \leq M} \e\biggl(\Bigl(k + \frac{\xi}{N} \Bigr)m^c \Bigr) \biggr)\biggr\rvert + \bigl\lvert \hat{\mu}_3(k) \bigr\rvert \Biggr).
    \end{equation*}
\end{Lemma}
\begin{proof}
By trigonometric identities and noting that $\lvert c_1 \rvert, \lvert c_2 \rvert \leq 1/2$ (in fact, these can be shown to go to zero as $M \to \infty$ by using a van der Corput argument similar to \cref{prop:estimate-sigma2-details}),
\begin{align*}
    \bigl\lvert \hat{\mu}_1(k) \bigr\rvert &= \frac{1}{1 + c_1} \biggl\lvert\mathop{\E}_{m \leq M} \bigl(\cos(2\pi km^c) + i \sin(2\pi k m^c)\bigr) \cos\biggl(\!-\frac{2\pi i \xi m^c}{N}\biggr) + \hat{\mu}_3(k) \biggr\rvert \\
    \begin{split}
        &\leq 2\Biggl(\Biggl\lvert \mathop{\E}_{m \leq M} \cos\biggl(2\pi \Bigl(k - \frac{\xi}{N} \Bigr) m^c \biggr)
        + \cos\biggl(2\pi \Bigl(k + \frac{\xi}{N}\Bigr) m^c \biggr)\\
        &\qquad\qquad + i\Biggl(\sin\biggl(2\pi \Bigl(k + \frac{\xi}{N}\Bigr) m^c \biggr) 
        + \sin\biggl(2\pi \Bigl(k - \frac{\xi}{N}\Bigr) m^c \biggr)\Biggr) \Biggr\rvert 
        + \bigl\lvert \hat{\mu}_3(k) \bigr\rvert \Biggr)
    \end{split}\\
    &\leq 2\Biggl(\biggl\lvert \mathop{\E}_{m \leq M} \e\biggl(\Bigl(k - \frac{\xi}{N}\Bigr) m^c \biggr)\biggr\rvert 
    + \biggl\lvert \mathop{\E}_{m \leq M} \e\biggl(\Bigl(k + \frac{\xi}{N}\Bigr) m^c \biggr)\biggr\rvert 
    + \bigl\lvert \hat{\mu}_3(k) \bigr\rvert \Biggr).
\end{align*}

Similarly, \begin{align*}
    \bigl\lvert \hat{\mu}_2(k) \bigr\rvert &= \frac{1}{1 + c_2} \biggl\lvert \mathop{\E}_{m \leq M} \bigl(\cos(2\pi km^c) + i \sin(2\pi k m^c) \bigr)
    \sin\biggl(\! -\frac{2\pi i \xi m^c}{N} \biggr) + \hat{\mu}_3(k) \biggr\rvert\\
    \begin{split}
        &\leq 2\Biggl(\Biggl\lvert\mathop{\E}_{m \leq M} \sin\biggl(2\pi \Bigl(k - \frac{\xi}{N} \Bigr)m^c \biggr)
        - \sin\biggr(2\pi \Bigl(k + \frac{\xi}{N} \Bigr)m^c \biggr)\\
        &\qquad\qquad + i \Biggl(\cos\biggl(2\pi \Bigl(k + \frac{\xi}{N} \Bigr) m^c \biggr) - \cos\biggl(2\pi \Bigl(k - \frac{\xi}{N} \Bigr)m^c \biggr)\Biggr) \Biggr\rvert 
        + \bigl\lvert \hat{\mu}_3(k) \bigr\rvert \Biggr)
    \end{split}\\ 
    &\leq 2 \Biggl(\biggl\lvert\mathop{\E}_{m \leq M} \e\biggl(\Bigl(k - \frac{\xi}{N} \Bigr) m^c \biggr) \biggr\rvert 
    + \biggl\lvert\mathop{\E}_{m \leq M} \e\biggl(\Bigl(k + \frac{\xi}{N}\Bigr)m^c \biggr) \biggr\rvert 
    + \bigl\lvert \hat{\mu}_3(k) \bigr\rvert \Biggr).
\end{align*}
\end{proof}

The following two lemmas put the previous computations together to bound the exponential sum $\mathop{\E}_{m \leq M} \e(-\xi \floor{m^c}/N)$ uniformly over $\xi$ in the desired range.

\begin{Lemma}\label{lem:expectation-to-measure}
    Let $\mu$ be defined as in \cref{lem:measure-stuff}. Then,
    \begin{equation}\label{eqn:turan-prep}
        \begin{split}
             \biggl\lvert \mathop{\E}_{m \leq M} &\1_{P}\bigl(\{m^c\}\bigr) \e\biggl(\! -\frac{\xi}{N} m^c \biggr) \biggr\rvert\\
             &\leq 5\Bigl(\bigl\lvert \mu_1(P) - \Leb(P) \bigr\rvert + \bigl\lvert \mu_2(P) - \Leb(P) \bigr\rvert + \bigl\lvert \mathop{\E}_{m \leq M} \delta_{\{m^c\}}(P) - \Leb(P) \bigr\rvert \Bigr).
         \end{split}
     \end{equation}
\end{Lemma}
\begin{proof}
    Noticing that
    \begin{equation*}
         \biggl\lvert \mathop{\E}_{m \leq M} \1_{P}\bigl(\{m^c\}\bigr) \e\biggl(\! -\frac{\xi}{N} m^c \biggr) \biggr\rvert 
         = \biggl\lvert \mathop{\E}_{m \leq M} \e\biggl(\! -\frac{\xi}{N} m^c \biggr) \delta_{\{m^c\}}(P) \biggr\rvert 
         = \bigl\lvert \mu(P) \bigr\rvert,
     \end{equation*} 
     the result is immediate from \cref{lem:measure-stuff}.
\end{proof}

Now we apply the \Erdos-\Turan{} inequality to each of the terms in \cref{eqn:turan-prep}, in view of the bounds for the Fourier transforms of $\hat{\mu}_1, \hat{\mu}_2$, and $\hat{\mu}_3$. 

\begin{Lemma}\label{lem:turan-bound}
    Let $K \in \Z^{+}$ be arbitrary. Then
    \begin{equation}\label{eqn:exp-sum}
        \begin{split}
            \sup_{\mkern-10mu \frac{N}{\log^b{N}} \leq \xi < \frac{N}{2}} &\biggl\lvert\mathop{\E}_{m \leq M} \e\biggl(\frac{\xi\floor{m^c}}{N} \biggr) \biggr\rvert\\
            &\qquad \leq \;\;\smashoperator[l]{\sup_{\mkern-10mu \frac{N}{\log^b{N}} \leq \xi < \frac{N}{2}}} \Biggl\lvert \frac{100}{\gamma} \biggl(\frac{1}{K} + \sum_{k = 1}^{K} \frac{\bigl\lvert \hat{\mu}_1(k) \bigr\rvert + \bigl\lvert \hat{\mu}_2(k) \bigr\rvert  + \bigl\lvert \hat{\mu}_3(k) \bigr\rvert}{k} \biggr) + 4\pi\frac{\xi}{N} \gamma \Biggr\rvert.
        \end{split}
    \end{equation}
\end{Lemma}
\begin{proof}
    Fix some $\xi$ satisfying $N/\log^b{N} \leq \xi < N/2$. In order of application, by \cref{lem:partition-bound}, by \cref{lem:expectation-to-measure}, by the fact that $\lvert \textbf{P}\rvert = 1/\gamma$, by the \Erdos-\Turan{} inequality (\cref{thm:erdos-turan}), and by the triangle inequality,
    \begin{align*}
        \biggl\lvert\mathop{\E}_{m \leq M} \e\biggl(\frac{\xi\floor{m^c}}{N} \biggr) \biggr\rvert 
        &\leq \sum_{P \in \textbf{P}}\: \biggl\lvert \mathop{\E}_{m \leq M} \1_{P}\bigl(\{m^c\}\bigr) \e\biggl(\! -\frac{\xi}{N} m^c \biggr) \biggr\rvert + 4\pi\frac{\xi}{N} \gamma\\
        \begin{split}
            &\leq 5\sum_{P \in \textbf{P}} \Bigl(\bigl\lvert \mu_1(P) - \Leb(P)\bigr\rvert + \bigl\lvert \mu_2(P) - \Leb(P) \bigr\rvert\\[-1.2em]
            &\hphantom{\leq 5\sum_{P \in \textbf{P}} \Bigl(\bigl\lvert\mu_1(P) - \Leb(P)\bigr\rvert}\;
            + \bigl\lvert \mathop{\E}_{m \leq M} \delta_{\{m^c\}}(P) - \Leb(P) \bigr\rvert \Bigr) + 4\pi\frac{\xi}{N}\gamma
        \end{split}\\ 
        &\leq \frac{50}{\gamma}\biggl(\frac{1}{K} + \sum_{k = 1}^{K} \frac{\bigl\lvert \hat{\mu}_1(k) \bigr\rvert + \bigl\lvert \hat{\mu}_2(k) \bigr\rvert  + \bigl\lvert \hat{\mu}_3(k) \bigr\rvert}{k} \biggr) + 4\pi\frac{\xi}{N} \gamma.
    \end{align*} 
    Then \cref{eqn:exp-sum} follows by taking suprema. 
\end{proof}

We will now bound the right-hand side of \cref{eqn:exp-sum}. Throughout, implicit constants in $\lesssim$ and $\approx$ depend only on $c$. First, note that the sum of absolute values of $\hat{\mu}_i$ is at most
\begin{equation*}
    10 \Biggl( \biggl\lvert \mathop{\E}_{m \leq M} \e\biggl( \Bigl(k - \frac{\xi}{N} \Bigr) m^c \biggr) \biggr\rvert 
    + \biggl\lvert \mathop{\E}_{m \leq M} \e\biggl( \Bigl(k + \frac{\xi}{N} \Bigr) m^c \biggr) \biggr\rvert 
    + \bigl\lvert \mathop{\E}_{m \leq M} \e(km^c) \bigr\rvert \Biggr),
\end{equation*}
in view of the bounds for the Fourier transforms we proved. We explicitly prove the bound for only the first term, but the proofs for the other two terms are analogous, as indicated in the proof of \cref{prop:estimate-sigma2-details}.

\begin{Corollary}\label{cor:bound-for-turan-bound}
    Let $\gamma = \frac{1}{\log^{A + 1} N}$ and
    \[
        U(\xi) = \frac{1}{K} + \sum_{k = 1}^{K} \frac{\Bigl\lvert \mathop{\E}_{m \leq M} \e\Bigl(\bigl(-\frac{\xi}{N} + k\bigr) m^c \Bigr) \Bigr\rvert}{k}.
    \]
    Then, 
    \begin{align*}
        \eqref{eqn:exp-sum}\; &\lesssim \;\,\smashoperator{\sup_{\mkern-10mu \frac{N}{\log^b N} \leq \xi < \frac{N}{2}}}\;\; \frac{U(\xi)}{\gamma} + \frac{\xi}{N} \gamma\\ 
        &\leq \;\,\smashoperator{\sup_{\mkern-10mu \frac{N}{\log^b{N}} \leq \xi < \frac{N}{2}}}\;\; U(\xi) \log^{A+1} N + \frac{1}{\log^{A + 1} N}.
    \end{align*}
\end{Corollary}

We will show that $\sup_{\xi} U(\xi) \lesssim 1/M^{t}$, where $t$ depends only on $c$, which will then give us the desired result that the left-hand side of \cref{eqn:exp-sum} $\lesssim 1/\log^{A + 1} N$. We divide the range of summation into dyadic intervals and use van der Corput bounds on each of them.

\begin{Proposition}\label{prop:estimate-sigma2-details}
    We have
    \[
        \sup_{\mkern-10mu \frac{N}{\log^b N} \leq \xi \leq \frac{N}{2}}
        \biggl\lvert \frac{1}{M} \sum_{m=1}^M \e\biggl(\!-\frac{\xi \floor{m^c}}{N} \biggr)\biggr\rvert 
        = \Oh\biggl(\frac{1}{\log^{A+1} N}\biggr).
    \]
\end{Proposition}

\begin{proof}
    For each $j \in \N$, let $S_j = [2^{j}, 2^{j + 1}) \cap [1, M]$. Then 
    \begin{equation}\label{eqn:dyadic-sum}
        \sum_{m \leq M} \e\biggl(\Bigl(\! -\frac{\xi}{N} + k \Bigr) m^c \biggr) 
        = \sum_j \sum_{m \in S_j} \e\biggl(\Bigl(\! -\frac{\xi}{N} + k \Bigr) m^c \biggr).
    \end{equation}

    For fixed $k$ and $\xi$, the inner sum is amenable to the van der Corput estimate \cref{thm:vdc-estimate}. Choose $X = 2^j$ and $q = \ceil{c - 2}$. Then \cref{eqn:derv-bound} is satisfied with $B_i$ depending only on $c$, and $F = (-\xi/N + k) X^c \approx kX^c$ . For the terms involving $\e(km^c)$ and $\e\bigl((k + \xi/N) m^c\bigr)$, the proof is identical, since $kX^c \approx (k + \xi/N) X^c$.
    Since \(X/F \lesssim 1\), we obtain
    \begin{equation*}
        \sum_{m \in S_j} \e\biggl(\Bigl(\! -\frac{\xi}{N} + k \Bigr) m^c \biggr)
        \lesssim k^\frac{1}{2^{q + 2} -2 } 2^{j \bigl(1 + \frac{c - q - 2}{2^{q + 2} - 2} \bigr)}.
    \end{equation*}
    
    Summing over $j$ up to $\floor*{\log_2 M}$ and noting that $j = 2^{j + 1} \approx M$, the upper limit, we obtain 
    \begin{equation*}
        \eqref{eqn:dyadic-sum} \lesssim k^{\frac{1}{2^{q + 2} - 2}} M^{1 + \frac{c - q - 2}{2^{q + 2} - 2}}.
    \end{equation*}
    Thus,
    \begin{align*}
        U(\xi) &\lesssim \frac{1}{K} + M^{\frac{c - q - 2}{2^{q + 2} - 2}} \sum_{k \leq K} k^{\frac{1}{2^{q + 2} - 2} - 1}\\
        &\lesssim \frac{1}{K} + M^{\frac{c - q - 2}{2^{q + 2} - 2}}K^{\frac{1}{2^{q + 2} - 2}}.
    \end{align*}
    Choosing $K$ so that the two summands above are equal, we see that $U(\xi) \lesssim M^\frac{c - q - 2}{2^{q + 2} - 1}$, independent of $\xi$.
\end{proof}

Now we can complete the proof. As noted in the beginning of the section, we factor out the exponential sum, use the bounds we obtained to decouple the product, and use Parseval's identity for the term involving $\lvert\hat{\Lambda}\rvert^2$.

\begin{Lemma}\label{lem:sigma2-proof-works}
    \cref{prop:estimate-sigma2-details} implies \cref{prop:estimate-sigma2}.
\end{Lemma}
\begin{proof}
    Indeed, by \cref{prop:estimate-sigma2-details} we obtain
    \begin{equation*}
        \sup_{\mkern-10mu \frac{N}{\log^b N} \leq \xi \leq \frac{N}{2}}
            \biggl\lvert \frac{1}{M} \sum_{m=1}^M \e\biggl(\! -\frac{\xi \floor{m^c}}{N} \biggr)\biggr\rvert 
            = \Oh\biggl(\frac{1}{\log^{A+1} N}\biggr).
    \end{equation*}
    Therefore,
    \begin{align*}
        \Biggl\lvert \sum_{\xi = \frac{N}{\log^b N}}^{\frac{N}{2}} \!\biggl\lvert \frac{\hat{\Lambda}(\xi)}{N} \biggr\rvert^2 \! \Biggl(\frac{1}{M} \sum_{m = 1}^M \e\biggl(\! -\frac{\xi \floor{m^c}}{N} \biggr)\Biggr)\Biggr\rvert
        &\leq \;\,\smashoperator[l]{\sup_{\mkern-11mu \frac{N}{\log^b N} \leq \xi \leq \frac{N}{2}}} \biggl\lvert \frac{1}{M} \sum_{m = 1}^{M} \e\biggl(\! -\frac{\xi \floor{m^c}}{N} \biggr)\biggr\rvert 
        \cdot \sum_{\mathclap{r = \frac{N}{\log^b N}}}^{\frac{N}{2}}\: \biggl\lvert \frac{\hat{\Lambda}(r)}{N} \biggr\rvert^2\\
        &\leq \;\,\smashoperator[l]{\sup_{\mkern-11mu \frac{N}{\log^b N} \leq \xi \leq \frac{N}{2}}} \biggl\lvert \frac{1}{M} \sum_{m = 1}^{M} \e\biggl(\! -\frac{\xi \floor{m^c}}{N} \biggr)\biggr\rvert 
        \cdot \sum_{r = 0}^{N - 1}\: \biggl\lvert \frac{\hat{\Lambda}(r)}{N} \biggr\rvert^2\\
        &= \;\,\smashoperator[l]{\sup_{\mkern-11mu \frac{N}{\log^b N} \leq \xi \leq \frac{N}{2}}} \biggl\lvert \frac{1}{M} \sum_{m = 1}^{M} \e\biggl(\! -\frac{\xi \floor{m^c}}{N} \biggr)\biggr\rvert 
        \cdot\frac{1}{N} \sum_{r = 0}^{N - 1} \bigl(\Lambda(r) \bigr)^2,
    \end{align*}
    where the last equality holds due to Parseval's identity (\cref{thm:parseval}).
    
    Now, since 
    \begin{equation*}
        \frac{1}{N} \sum_{\xi = 0}^{N - 1} \bigl(\Lambda(\xi)\bigr)^2 
        \leq \frac{1}{N} \sum_{\xi = 0}^{N - 1} \Lambda(\xi) \cdot \log N,
    \end{equation*}
    we have by the Prime Number Theorem that $\frac{1}{N} \sum_{\xi = 0}^{N - 1} \bigl(\Lambda(\xi) \bigr)^2 = \Oh(\log N)$. Therefore,
    \begin{align*}
        \sup_{\mkern-10mu \frac{N}{\log^b N} \leq \xi \leq \frac{N}{2}}
        \biggl\lvert \frac{1}{M} \sum_{m = 1}^{M} \e\biggl(\! -\frac{\xi \floor{m^c}}{N} \biggr)\biggr\rvert 
        \cdot \frac{1}{N}\sum_{r = 0}^{N - 1} \bigl(\Lambda(r)\bigr)^2 
        &= \Oh\biggl(\frac{1}{\log^{A + 1} N} \biggr) \Oh(\log N)\\
        &= \Oh\biggl(\frac{1}{\log^A N} \biggr),
    \end{align*}
    as required. Combining the above with \cref{prop:estimate-sigma1}, we have 
    \begin{equation*}
        \sum_{\xi = 1}^{\frac{N}{2}} \biggl\lvert \frac{\hat{\Lambda}(\xi)}{N} \biggr\rvert^2 \Biggl(\frac{1}{M} \sum_{m = 1}^M \e\biggl(\! -\frac{\xi \floor{m^c}}{N} \biggr)\Biggr)
        = \Oh \biggl(\frac{1}{\log^A N}\biggr),
    \end{equation*}
    and thus, by symmetry,
    \begin{equation*}
        \sum_{\xi = 1}^{N-1} \biggl\lvert \frac{\hat{\Lambda}(\xi)}{N} \biggr\rvert^2 \Biggl(\frac{1}{M} \sum_{m = 1}^M \e\biggl(\! -\frac{\xi \floor{m^c}}{N} \biggr)\Biggr) 
        = \Oh\biggl(\frac{1}{\log^A N}\biggr).
    \end{equation*}
\end{proof}

\section{Further Directions}
\hspace{20pt} In this section we give some possible extensions of our result. As mentioned in the introduction, the case of \(1 < c < 2\) remains open. However, we not only expect our result to hold for this case, but also for all fractional polynomials.
\begin{Conjecture}\label{conj:frac} 
    Let $f(x) = a_1 x^{c_1} + a_2 x^{c_2} + \ldots + a_k x^{c_k}$, $A > 1$ be given, $c_1 > 0$ a non-integer, $a_1 \neq 0$, and $c_1 > c_2 > \ldots > c_k$. Then there exists $B = B(f,A)$ such that 
    \[
        \mathop{\E}_{m \leq M} \mathop{\E}_{n \leq N} \Lambda(n) \Lambda\bigl(n + \floor{f(x)}\bigr) 
        = 1 + \Oh\biggl(\frac{1}{\log^A N} \biggr),
    \] 
    where $M = \frac{N^{1/c_1}}{\log^{B} N}$.
\end{Conjecture}

Attacking this using our method requires a Vinogradov mean value theorem for fractional polynomials, of the type proven in \cite{poulias2021}. We believe this should be technically demanding but possible. Moreover, we actually expect \cref{conj:frac} to hold for functions from a Hardy field satisfying some growth conditions, but formulating this would be tedious so we refrain from doing so.

Secondly, we expect a similar result to hold for longer progressions.

\begin{Conjecture}
    Let $f_1, f_2, \ldots, f_k$ satisfy the same assumptions as $f$ in \cref{conj:frac}, let $A > 1$, and let $d$ be the largest of the degrees of the $f_i$'s. Then there exists $B = B(f_1, \ldots, f_k, A)$ such that 
    \[
        \mathop{\E}_{m \leq M} \mathop{\E}_{n \leq N} \prod_{i = 1}^k \Lambda\bigl(n + \floor{f_i(m)}\bigr) 
        = 1 + \Oh\biggl(\frac{1}{\log^A N} \biggr),
    \] 
    where $M = \frac{N^{1/d}}{log^{B} N}$.
\end{Conjecture}

We expect this to be much harder to prove, and should rely on methods of higher order Fourier analysis. 

Finally, the following problem posed by Frantzikinakis generalizes \cref{cor:infinitely-many-pairs} in two directions; the first statement extends it to arithmetic progressions with common difference in the image of a fractional power, and the second generalizes our result to more than one fractional power.

\begin{Conjecture}[\protect{\cite[Problem 30]{frantzikinakis2016}}]
Let $k \in \N$ and $c$ be a positive real number. Then, for infinitely many pairs $(p, m) \in \P \times \N$, we have $\{p, p + \floor{m^c}, p + 2\floor{m^c}, \ldots, p + k\floor{m^c}\} \subseteq \P$. 

Further, let $c_1, \ldots, c_k$ be positive real numbers. Then, for infinitely many pairs $(p, m) \in \P \times \N$, we have $\{ p, p + \floor{m^{c_1}}, p + \floor{m^{c_2}}, \ldots, p + \floor{m^{c_k}} \} \subseteq \P$.
\end{Conjecture}

\printbibliography

\bigskip
\footnotesize
\noindent
Bora \c{C}al{\i}m\\
\textsc{Ko\c{c} University}\\
\href{mailto:bcalim21@ku.edu.tr}
{\texttt{bcalim21@ku.edu.tr}}
\\ \\
Ioannis Iakovakis\\
\textsc{Aristotle University of Thessaloniki}\\
\href{mailto:iiakovak@math.auth.gr}
{\texttt{iiakovak@math.auth.gr}}
\\ \\
Sophie Long\\
\textsc{Queen Mary University of London}\\
\href{mailto:ah23211@qmul.ac.uk}
{\texttt{ah23211@qmul.ac.uk}}
\\ \\
Jack Moffatt\\
\textsc{Yale University}\\
\href{mailto:jack.moffatt@yale.edu}
{\texttt{jack.moffatt@yale.edu}}
\\ \\
Deborah Wooton\\
\textsc{University of Utah}\\
\href{mailto:deborah.wooton@utah.edu}
{\texttt{deborah.wooton@utah.edu}}

\end{document}